\documentclass[10pt]{amsart}
\usepackage{latexsym,amsmath,amssymb}
\usepackage{mathrsfs}
\usepackage{mathabx}
 \usepackage{indentfirst}
 \renewcommand{\epsilon}{\varepsilon}
 \newcommand{\newsection}[1]
  {\section{#1}\setcounter{theorem}{0} \setcounter{equation}{0}\par\noindent}

   \newtheorem{theorem}{Theorem}[section]
   \newtheorem{lemma}[theorem]{Lemma}
 \newtheorem{corr}[theorem]{Corollary}
 
 \newtheorem{proposition}[theorem]{Proposition}
 \newtheorem{deff}[theorem]{Definition}
 \newtheorem{remark}[theorem]{Remark}

  \numberwithin{equation}{section}

 \newcommand{\bth}{\begin{theorem}}
 \newcommand{\ble}{\begin{lemma}}
 \newcommand{\bcor}{\begin{corr}}
 \newcommand{\bdeff}{\begin{deff}}
 \newcommand{\bprop}{\begin{proposition}}
 \def\be{\begin{equation}}
\def\ee{\end{equation}}
\def\bt{\begin{theorem}}
\def\et{\end{theorem}}
\def\ba{\begin{array}}
\def\ea{\end{array}}
\def\bl{\begin{lemma}}
\def\el{\end{lemma}}

 \newcommand{\ele}{\end{lemma}}
 \newcommand{\ecor}{\end{corr}}
 \newcommand{\edeff}{\end{deff}}
 
 \newcommand{\eprop}{\end{proposition}}

 \newcommand{\eps}{\varepsilon}

 \renewcommand{\Pi}{\varPi}

 \renewcommand{\epsilon}{\varepsilon}

  \newcommand{\R}{{\mathbb R}}

\title[Global existence of semilinear wave equations with damping] { Global existence for semilinear wave equations with scaling invariant damping in 3-D}

\date{\today}

\begin{document}
\maketitle

\centerline{
\author{
Ning-An Lai*
  \footnote{*Corresponding Author: Institute of Nonlinear Analysis and Department of Mathematics, Lishui University, China.
 {\it Email: ninganlai@lsu.edu.cn}}
 }
 \and
Yi Zhou
  \footnote{ School of Mathematical Sciences, Fudan University, Shanghai, China. and Department of Mathematical Sciences, Jinan University, Guangzhou 510632, China.
 {\it Email: yizhou@fudan.edu.cn}
 }
  }

\begin{abstract}

Global existence for small data Cauchy problem of semilinear wave equations with scaling invariant damping in 3-D is established in this work, assuming that the data are radial and the constant in front of the damping belongs to $[1.5, 2)$. The proof is based on a weighted $L^2-L^2$ estimate for inhomogeneous wave equation, which is established by interpolating between energy estimate and Morawetz type estimate.

\end{abstract}
{\bf Keywords: semilinear wave equation, scaling invariant damping, global existence, weighted $L^2-L^2$ estimate}

{\bf MSC2020:} {35L71, 35L05, 35B40}

\newsection{Introduction}

This paper is devoted to studying radial solution of the semilinear wave equation with scaling invariant damping in 3-D
\begin{equation} \label{scalwave}
\begin{cases}
\varphi_{tt}-\varphi_{rr}-\frac{2\varphi_r}{r}+\frac{\mu
\varphi_t}{t+2}=|\varphi|^p,\\
t=0:\, \varphi=\eps\varphi_0(r),\varphi_t=\eps\varphi_1(r),
\end{cases}
\end{equation}
where $\eps$ denotes the smallness of the initial data. The nonnegative initial data come from the energy space and have compact support
\begin{equation}\label{supp}
\begin{aligned}
supp~\varphi_0(r), \varphi_1(r)\subset\{r\big|r\le 1\}.
\end{aligned}
\end{equation}

This kind of semilinear wave equation with time dependent variable coefficients in front of the damping has been widely studied recently, and the general model is
\begin{equation}\label{semimodel}
\begin{aligned}
\Phi_{tt}-\Delta\Phi+\frac{\mu}{(1+t)^\beta}\Phi_t=|\Phi|^p.
\end{aligned}
\end{equation}
According to the asymptotic behavior of the solution of the corresponding linear equation, we have four different type dampings, thus
\begin{center}
\begin{tabular}{|c|c|c|}
\hline
$\beta\in (-\infty, -1)$ & overdamping &
\begin{tabular}{c}
solution does not \\
decay to zero
\end{tabular}
\\
\hline
$\beta\in [-1, 1)$ & effective &
\begin{tabular}{c}
solution behaves like\\
that of heat equation
\end{tabular}
\\
\hline
$\beta=1$ &
\begin{tabular}{c}
scaling invariant\\
weak damping
\end{tabular} &
\begin{tabular}{c}
the asymptotic behavior\\depends on $\mu$
\end{tabular}
\\
 \hline
$\beta\in(1,\infty)$ & scattering &
\begin{tabular}{c}
solution behaves like
that\\ of wave equation without damping
\end{tabular}\\
\hline
\end{tabular}
\end{center}
We refer to the works \cite{HO, MN, N1, Wir1, Wir2, Wir3} and the conclusive table in \cite{LT20}. Based on this classification, people try to figure out the critical power($p_c(n)$, $n$ is the dimension) for \eqref{semimodel} for each case. Here \lq\lq critical power" denotes the threshold value of $p$ which divides the problem into blow-up and global existence parts. If $\beta<-1$, there is global solution for all $p>1$, see \cite{IWnew}. If $\beta\in [-1, 1)$, it has been showed that the critical power is exactly the same as Fujita number, i.e., $p_c(n)=p_F(n)=1+\frac 2n$, see \cite{DLR13, FIW, ISW, LaZ, LZ, LNZ12, TY, WY17}. If $\beta>1$, due to the blow-up results in \cite{LT, WY, Wak14} for $1<p<p_S(n)$ and $n\ge 1$ and global existence results in \cite{LW} for $p>p_S(n)$ and $n=3, 4$, we may conjecture the critical power is Strauss exponent, which is the critical power for semilinear wave equation without damping and the positive root of the quadratic equation
\begin{equation}
\label{quadratic}
\gamma(p,n):=2+(n+1)p-(n-1)p^2=0.
\end{equation}
The case of $\beta=1$, which we mean the scaling invariant damping, has attracts more and more attention, due to the reason that the critical power also depends on the size of the constant $\mu$ in front of the damping. Roughly speaking, if $\mu$ is relatively large, the critical power will be Fujita type while if $\mu$ is relatively small the critical power will be Strauss type. According to the results in \cite{DABI, DL1, WY14_scale} we know the critical power is $p_F(n)$ at least for
 \[
 \mu\geq
 \left\{
 \begin{array}{cl}
5/3 & \mbox{for}\ n=1,\\
3 & \mbox{for}\ n=2,\\
n+2 & \mbox{for}\ n\geq 3.
\end{array}
\right.
\]
If $\mu=2$, it is interesting to see that the equation can be changed into a one without damping by a transformation, and due to \cite{D-L, DLR15, KS, Lai, Pa, wak16}, we now know that the critical power is
\[
p_c(n)=\max\{p_F(n), p_S(n+2)\}.
\]
If $\mu\neq 2$, the Strauss type blow-up result was first established in \cite{LTW}, which was improved by \cite{IS}. We refer the reader to \cite{LT20, N2} for more detailed introduction for the related results.

However, till the moment, there is no global result for $\mu\neq 2$ and
\[
0<\mu<\frac{n^2+n+2}{n+2}~and~p> p_S(n+\mu),
\]
which means that in this case the critical power($p=p_S(n+\mu)$) is still unfixed. In this paper, we will show global existence for \eqref{scalwave} in $\mathbb{R}^3$ for
\[
1.5\le\mu<2~and~ p_S(3+\mu)<p\le 2,
\]
and this result will confirm the critical power for some range of $\mu$ in $\R^3$. The proof is quite elementary, and the key step is to establish a weighted $L^2-L^2$ estimate by interpolating between an energy estimate and a Morawetz type estimate.
\begin{remark}
The similar idea has been used in \cite{Lai} to show the global existence of non-radial solution for \eqref{scalwave} with $\mu=0, p>p_S(3)$ and $\mu=2, p>p_S(5)$ in $\mathbb{R}^3$.
\end{remark}

The main result is as follows.
\begin{theorem} \label{thm1.1}
Let $1.5\le\mu<2$ and $p_S(n+\mu)<p\le 2$. And $\varepsilon$ represents the smallness of the initial data. Assuming the support of the initial data satisfy \eqref{supp}. Then there exists a positive constant $\eps_0$ such that if $0<\eps<\eps_0$, problem \eqref{scalwave} has global solution.
\end{theorem}

\section{Weighted $L^2-L^2$ estimate for inhomogeneous wave equation}

We first take the transformation
\[
\psi(t, r)=(t+2)^{\frac{\mu}{2}}\varphi(t, r),
\]
then $\psi(t, r)$ satisfies
\begin{equation}\label{1.1}
\left \{
\begin{aligned}
&\psi_{tt}-\psi_{rr}-\frac{2\psi_r}{r}+\frac{\mu(2-\mu)\psi}{4{(t+2)}^2}
=\frac{{|\psi|}^p}{{(t+2)}^{\frac{\mu(p-1)}{2}}},\\
&\psi(0, r)=2^{\frac{\mu}{2}}\eps\varphi_0(r), \psi_t(0, r)=\eps\Big \{\frac{\mu}{2}2^{\frac{\mu}{2}-1}\varphi_0(r)+
2^{\frac{\mu}{2}}\varphi_1(r)\Big\}.
\end{aligned} \right.
\end{equation}

Let
\[
u=\frac{t+2+r}{2}, \overline{u}=\frac{t+2-r}{2}
\]
and
\[
\phi(u, \overline{u})=(u-\overline{u})\psi(u+\overline{u}-2, u-\overline{u}),
\]
then $\phi$ satisfies the following system
\begin{equation*}
\begin{cases}
\phi_{u\bar{u}}+\frac{\mu(2-\mu)\phi}{4{(u+\bar{u})}^2}=\frac{{|\phi|}^p}
{{(u-\bar{u})}^{p-1}{(u+\bar{u})}^{\frac{\mu(p-1)}{2}}},\\
t=0:\, \phi=r\psi_0,\phi_t=r\psi_1.
\end{cases}
\end{equation*}

Next we are going to establish the weighted $L^2-L^2$ estimate for the following inhomogeneous second order partial differential equation
\begin{equation}\label{inhom}
\begin{cases}
\phi_{u\bar{u}}+\frac{\mu(2-\mu)\phi}{4{(u+\bar{u})}^2}=G(u,\bar{u}), \\
t=0:\, \phi=r\psi_0,\phi_t=r\psi_1.
\end{cases}
\end{equation}
First we show a standard energy estimate for \eqref{inhom}.
\begin{lemma}[energy estimate] \label{lem1}
Let $\phi(u, \bar{u})$ solve \eqref{inhom}, and $\bar{U}$ be a positive constant, then we have
\begin{equation}\label{energy}
\begin{aligned}
&\sup\limits_{\frac{1}{2}\leq \bar{u} \leq \bar{U}} {\left(\int_{\max (\bar{u},2-\bar{u})}^{+\infty} {\phi_u}^2 du\right)}^{\frac{1}{2}} \\
\lesssim &{\eps\left( {\| \psi_0 \|}^2_{H^1(\mathbb{R}^3)}+{\| \psi_1 \|}^2_{L^2(\mathbb{R}^3)}\right)}^{\frac{1}{2}}+\int_{\frac{1}{2}}^{\bar{U}} {\left(\int_{\max (\bar{u},2-\bar{u})}^{+\infty} G^2 du \right)}^{\frac{1}{2}}d\bar{u}.
\end{aligned}
\end{equation}
\end{lemma}
\begin{proof}
Multiplying the equation in \eqref{inhom} by $\phi_u$, we get
$$
\partial_{\bar{u}} \frac{{\phi_u}^2}{2}+\frac{\mu(2-\mu){\partial_u \phi}^2}{8{(u+\bar{u})}^2}=\phi_u G,
$$
and then integrating it with respect to $u$ over $[\max (\bar{u},2-\bar{u}), +\infty)$ one has
\begin{eqnarray*}
&&\partial_{\bar{u}} \int_{\max (\bar{u},2-\bar{u})}^{+\infty} \frac{{\phi_u}^2}{2} du+\mu(2-\mu)\int_{\max (\bar{u},2-\bar{u})}^{+\infty} \frac{\phi^2}{4{(u+\bar{u})}^3} \\
&&=-\frac{\psi^2(t, 0)}{2}+\frac{\eps}{2}\left[\psi_0(r)+r\partial_r\psi_0(r)+r\psi_1(r)\right]+\int_{\max (\bar{u},2-\bar{u})}^{+\infty} \phi_u G du.\\
\end{eqnarray*}
Integrating the above equality with respect to $\bar{u}$ over $[\frac12, \bar{U}]$ yields
\begin{equation}\label{ener2}
\begin{aligned}
&\sup\limits_{\frac{1}{2}\leq \bar{u} \leq \bar{U}} \int_{\max (\bar{u},2-\bar{u})}^{+\infty} {\phi_u}^2 du \\
\leq &C\eps\left( {\|\psi_0 \|}^2_{H^1(\mathbb{R}^3)}+{\| \psi_1 \|}^2_{L^2(\mathbb{R}^3)}\right)+\int_{\frac{1}{2}}^{\bar{U}} {\left(\int_{\max (\bar{u},2-\bar{u})}^{+\infty} {\phi_u}^2 du \right)}^{\frac{1}{2}}{\left(\int_{\max (\bar{u},2-\bar{u})}^{+\infty} G^2 du \right)}^{\frac{1}{2}}d\bar{u},
\end{aligned}
\end{equation}
where we used the fact $0<\mu(2-\mu)<1$ and the solution vanishes when $\bar{u}=\frac12$. And hence the energy estimate \eqref{energy} follows.
\end{proof}

On the other hand, we may establish a Morawetz type estimate for \eqref{inhom}.
\begin{lemma}[Morawetz type estimate] \label{lem2}
Let $\phi(u, \bar{u})$ solve \eqref{inhom}, and $\bar{U}$ be a positive constant, then we have
\begin{equation}\label{Mora}
\begin{aligned}
&\sup\limits_{\frac{1}{2}\leq \bar{u} \leq \bar{U}} {\left(\int_{\max (\bar{u},2-\bar{u})}^{+\infty} u^3(u-\bar{u}){\phi_u}^2 du\right)}^{\frac{1}{2}} \\
\lesssim &\eps{\left( {\|\psi_0 \|}^2_{H^1(\mathbb{R}^3)}+{\| \psi_1 \|}^2_{L^2(\mathbb{R}^3)}\right)}^{\frac{1}{2}}+\int_{\frac{1}{2}}^{\bar{U}} {\left(\int_{\max (\bar{u},2-\bar{u})}^{+\infty} u^3(u-\bar{u})G^2 du \right)}^{\frac{1}{2}}d\bar{u}
\end{aligned}
\end{equation}
\end{lemma}
\begin{proof}
Multiplying the equation in \eqref{inhom} with $(u-\bar{u})\phi_{u}$ yields
$$
\partial_{\bar{u}} \left[(u-\bar{u}) \frac{{\phi_u}^2}{2}\right]+\frac{{\phi_u}^2}{2}+\frac{\mu(2-\mu)(u-\bar{u}){\partial_u \phi}^2}{8{(u+\bar{u})}^2}=(u-\bar{u})\phi_u G,
$$
integrating which with respect to $u$ over $[\max (\bar{u},2-\bar{u}), +\infty)$ we come to
\begin{equation}\label{Mora1}
\begin{aligned}
&\partial_{\bar{u}} \int_{\max (\bar{u},2-\bar{u})}^{+\infty} u^3(u-\bar{u}) \frac{{\phi_u}^2}{2} du+\frac{1}{2}\int_{\max (\bar{u},2-\bar{u})}^{+\infty} u^3 {\phi_u}^2 du \\
&-\mu(2-\mu)[\int_{\max (\bar{u},2-\bar{u})}^{+\infty}\frac{(\partial_u [u^3(u-\bar{u})](u+\bar{u})-2u^3(u-\bar{u}))}{8{(u+\bar{u})}^3} \phi^2 du ]\\
=&\frac{r(r+2)^3}{16}\left[\psi_0(r)+r\psi_1(r)+r\partial_r\psi_0(r)\right]^2
+\frac{\mu(2-\mu)}{256}r^3(r+2)^3\psi_0^2(r)\\
&+\int_{\max (\bar{u},2-\bar{u})}^{+\infty} u^3(u-\bar{u})\phi_u G du.\\
\end{aligned}
\end{equation}
A direct calculation shows that
\begin{eqnarray*}
&&\partial_u [u^3(u-\bar{u})](u+\bar{u})-2u^3(u-\bar{u}) \\
&&=(3u^2(u-\bar{u})+u^3)(u+\bar{u})-2u^3(u-\bar{u}) \\
&&=u^2\{(3r+u)(2u-r)-2ur\} \\
&&=u^2\{2u^2-3r^2+3ur\}\\
&&=u^2\{2u^2+\frac{3}{4}u^2-3(r^2-ur+\frac{u^2}{4})\}\\
&&\leq (2+\frac{3}{4})u^4 \leq (2+\frac{3}{4})u{(t+2)}^3,
\end{eqnarray*}
therefore, we obtain from \eqref{Mora1}
\begin{equation}\label{Mora2}
\begin{aligned}
&\frac{1}{2}\partial_{\bar{u}}\int_{\max (\bar{u},2-\bar{u})}^{+\infty}u^3(u-\bar{u})\phi_u^2 du+\frac{1}{2}\int_{\max (\bar{u},2-\bar{u})}^{+\infty}u^3 {\phi_u}^2 du \\
&-\frac{\mu(2-\mu)}{2}[\int_{\max (\bar{u},2-\bar{u})}^{+\infty}\frac{1}{4}(2+\frac{3}{4})u \phi^2 du] \\
\lesssim &\frac{r(r+2)^3}{16}\left[\psi_0(r)+r\psi_1(r)+r\partial_r\psi_0(r)\right]^2
+\frac{\mu(2-\mu)}{256}r^3(r+2)^3\psi_0^2(r)\\
&+\int_{\max (\bar{u},2-\bar{u})}^{+\infty} u^3(u-\bar{u})\phi_u G du.\\
\end{aligned}
\end{equation}
Now we claim a Hardy type inequality
\begin{equation}\label{hardy}
\int_{\max (\bar{u},2-\bar{u})}^{+\infty} \phi^2 u du \leq -\frac{r^2(r+2)^2}{4}\psi_0^2(r)+\int_{\max (\bar{u},2-\bar{u})}^{+\infty} u^3 {\phi_u}^2 du.
\end{equation}
Hence we may get \eqref{Mora} by combining \eqref{Mora2} and  claim \eqref{hardy}, and then integrating with respect to $\bar{u}$ over $[\frac12, \bar{U}]$, since we have
\[
\frac12--\frac{\mu(2-\mu)}{2}\frac{1}{4}(2+\frac{3}{4})>0
\]
for $0<\mu<2$. We are left with the proof of claim \eqref{hardy}. By integration by parts, one has
\begin{equation}\label{Hardyproof1}
\begin{aligned}
\int_{\max (\bar{u},2-\bar{u})}^{+\infty} \phi^2 u du=&-\left(\frac{r+2}{2}\right)^2
\left(r\psi_0(r)\right)^2-\int_{\max (\bar{u},2-\bar{u})}^{+\infty}u(2\phi\phi_uu+\phi^2)du,\\
\end{aligned}
\end{equation}
which implies
\begin{equation}\label{Hardyproof2}
\begin{aligned}
&2\int_{\max (\bar{u},2-\bar{u})}^{+\infty} \phi^2 u du\\
=&-\left(\frac{r+2}{2}\right)^2
\left(r\psi_0(r)\right)^2-2\int_{\max (\bar{u},2-\bar{u})}^{+\infty}\phi\phi_uu^2du\\
\le& -\left(\frac{r+2}{2}\right)^2
\left(r\psi_0(r)\right)^2+2\left(\int_{\max (\bar{u},2-\bar{u})}^{+\infty}\phi^2udu\right)^{\frac12}\left(\int_{\max (\bar{u},2-\bar{u})}^{+\infty}u^3\phi_u^2du\right)^{\frac12},
\end{aligned}
\end{equation}
and this in turn gives \eqref{hardy}.
\end{proof}

Interpolating between \eqref{energy} in Lemma \ref{lem1} and \eqref{Mora} in Lemma \ref{lem2}, we get
\begin{lemma}\label{lem3}
Let $\phi(u, \bar{u})$ solve \eqref{inhom}, and $\bar{U}$ be a positive constant, then for
$0\leq s\leq 1$, there holds
\begin{equation}\label{Weig}
\begin{aligned}
&\sup\limits_{\frac{1}{2}\leq \bar{u} \leq \bar{U}} \left(\int_{\max (\bar{u},2-\bar{u})}^{+\infty} u^{3s}{(u-\bar{u})}^s{\phi_u}^2 du \right)^{\frac12}\\
\lesssim &{\eps\left( {\| \psi_0 \|}^2_{H^1(\mathbb{R}^3)}+{\| \psi_1 \|}^2_{L^2(\mathbb{R}^3)}\right)}^{\frac{1}{2}}+\int_{\frac{1}{2}}^{\bar{U}} {\left(\int_{\max (\bar{u},2-\bar{u})}^{+\infty} u^{3s}{(u-\bar{u})}^s G^2 du \right)}^{\frac{1}{2}}d\bar{u}.
\end{aligned}
\end{equation}
\end{lemma}

\section{Sobolev type inequalities}

In this section, we shall prove several Sobolev type inequalities. In the following we take $s=\frac{1}{4}+\frac{1}{2p}$ and denote
$$
M(\phi)(\bar{u})={\left(\frac{1}{2} \int_{\max (\bar{u},2-\bar{u})}^{+\infty} u^{3s}{(u-\bar{u})}^s{\phi_u}^2 du \right)}^{\frac{1}{2}}.
$$
\begin{lemma} \label{lem4}
For $\bar{u} \geq 1$, it holds that
\begin{equation}\label{sob1}
\sup\limits_{u} u^{\frac{2}{p}} \phi^2(u,\bar{u}) \lesssim {\left(M(\phi)(\bar{u})\right)}^2
\end{equation}
\end{lemma}
\begin{proof}
Direct computation implies that
\begin{equation}\label{Sobstep1}
\begin{aligned}
u^{\frac{2}{p}} \phi^2(u,\bar{u})=&-u^{\frac{2}{p}}\int_u^{+\infty} \partial_{\lambda} \phi^2(\lambda,\bar{u}) d\lambda \\
\leq& 2 u^{\frac{2}{p}} \int_u^{+\infty} |\phi_\lambda| |\phi| d\lambda \\
\leq& 2 \int_u^{+\infty} \lambda^{\frac{2}{p}} |\phi_\lambda| |\phi| d\lambda \\
\leq&{\left( \int_{\max(\bar{u},2-\bar{u})}^{+\infty} u^{3s} {(u-\bar{u})}^s {\phi_u}^2 du \right)}^{\frac{1}{2}}\\
&\times{\left( \int_{\max(\bar{u},2-\bar{u})}^{+\infty} u^{\frac{5}{2p}-\frac{3}{4}} {(u-\bar{u})}^{-s} {\phi}^2 du \right)}^{\frac{1}{2}}.
\end{aligned}
\end{equation}
Since $1-s>0, \frac{5}{2p}-\frac{3}{4}>0$, then for $\bar{u}\geq 1$ we have
\begin{equation}\label{Sobstep2}
\begin{aligned}
&\int_{\max(\bar{u},2-\bar{u})}^{+\infty} u^{\frac{5}{2p}-\frac{3}{4}} {(u-\bar{u})}^{-s} {\phi}^2 du =\int_{\bar{u}}^{+\infty} u^{\frac{5}{2p}-\frac{3}{4}} {(u-\bar{u})}^{-s} {\phi}^2 du\\
=&\frac{1}{1-s}\int_{\bar{u}}^{+\infty} u^{\frac{5}{2p}-\frac{3}{4}}  {\phi}^2 d{(u-\bar{u})}^{1-s}\\
=&-\frac{1}{1-s}\int_{\bar{u}}^{+\infty} (\frac{5}{2p}-\frac{3}{4})u^{\frac{5}{2p}-\frac{7}{4}}  {\phi}^2(u-\bar{u})^{1-s}du \\ &-\frac{2}{1-s} \int_{\bar{u}}^{+\infty} u^{\frac{5}{2p}-\frac{3}{4}}  {(u-\bar{u})}^{1-s} \phi \phi_u du\\
\leq &\frac{2}{1-s} \int_{\bar{u}}^{+\infty} u^{\frac{5}{2p}-\frac{3}{4}}  {(u-\bar{u})}^{1-s} |\phi| |\phi_u| du \\
\lesssim &{\left( \int_{\bar{u}}^{+\infty} u^{\frac{5}{2p}-\frac{3}{4}}  {\phi}^2{(u-\bar{u})}^{-s} du\right)}^{\frac{1}{2}}{\left(\int_{\bar{u}}^{+\infty} u^{\frac{5}{2p}-\frac{3}{4}}{(u-\bar{u})}^{2-s}  {\phi_u}^2 du \right)}^{\frac{1}{2}},
\end{aligned}
\end{equation}
which implies
\begin{equation}\label{Sobstep3}
\begin{aligned}
&\int_{\max(\bar{u},2-\bar{u})}^{+\infty} u^{\frac{5}{2p}-\frac{3}{4}} {(u-\bar{u})}^{-s} {\phi}^2 du \\
\leq &\int_{\bar{u}}^{+\infty} u^{\frac{5}{2p}-\frac{3}{4}}{(u-\bar{u})}^{2-s}  {\phi_u}^2 du\\
\leq &\int_{\bar{u}}^{+\infty} u^{\frac{5}{2p}-\frac{3}{4}+2-2s}{(u-\bar{u})}^{s}  {\phi_u}^2 du\\
&\lesssim {\left(M(\phi)(\bar{u}) \right)}^2,
\end{aligned}
\end{equation}
where we have used the fact that $2-s \geq s$ and hence
$${(u-\bar{u})}^{2-s}\leq u^{2-2s}{(u-\bar{u})}^s,$$
and Lemma \ref{lem4} follows.
\end{proof}
In a similar way, we can prove
\begin{lemma} \label{lem8}
For $\bar{u} \le 1$, it holds that
\begin{equation}\label{sob2}
\sup\limits_{u} u^{\frac{2}{p}} \phi^2(u,\bar{u}) \lesssim {\left(M(\phi)(\bar{u})\right)}^2+(1-\bar{u})^{3-s}\psi_0^2(2-2\bar{u}).
\end{equation}
\end{lemma}
\begin{proof}
Since $u\ge 2-\bar{u}$, then as in the proof of the last lemma we have
\begin{equation}\label{Sob2step1}
\begin{aligned}
u^{\frac{2}{p}} \phi^2(u,\bar{u})=&-u^{\frac{2}{p}}\int_u^{+\infty} \partial_{\lambda} \phi^2(\lambda,\bar{u}) d\lambda \\
\leq& 2 u^{\frac{2}{p}} \int_u^{+\infty} |\phi_\lambda| |\phi| d\lambda \\
\leq& 2 \int_u^{+\infty} \lambda^{\frac{2}{p}} |\phi_\lambda| |\phi| d\lambda \\
\leq&{\left( \int_{2-\bar{u}}^{+\infty} u^{3s} {(u-\bar{u})}^s {\phi_u}^2 du \right)}^{\frac{1}{2}}\\
&\times{\left( \int_{2-\bar{u}}^{+\infty} u^{\frac{5}{2p}-\frac{3}{4}} {(u-\bar{u})}^{-s} {\phi}^2 du \right)}^{\frac{1}{2}}.
\end{aligned}
\end{equation}
Since $1-s>0, \frac{5}{2p}-\frac{3}{4}>0$, then for $\bar{u}\le 1$ we have
\begin{equation}\label{Sob2step2}
\begin{aligned}
&\int_{2-\bar{u}}^{+\infty} u^{\frac{5}{2p}-\frac{3}{4}} {(u-\bar{u})}^{-s} {\phi}^2 du\\
=&\frac{1}{1-s}\int_{2-\bar{u}}^{+\infty} u^{\frac{5}{2p}-\frac{3}{4}}  {\phi}^2 d{(u-\bar{u})}^{1-s}\\
\le&\frac{1}{1-s}\left((1-\bar{u})^{1-s}\phi^2(2-\bar{u}, \bar{u}) -2\int_{2-\bar{u}}^{+\infty} u^{\frac{5}{2p}-\frac{3}{4}}  {(u-\bar{u})}^{1-s} \phi \phi_u du\right)\\
\lesssim &{\left( \int_{2-\bar{u}}^{+\infty} u^{\frac{5}{2p}-\frac{3}{4}}  {\phi}^2{(u-\bar{u})}^{-s} du\right)}^{\frac{1}{2}}{\left(\int_{2-\bar{u}}^{+\infty} u^{\frac{5}{2p}-\frac{3}{4}}{(u-\bar{u})}^{2-s}  {\phi_u}^2 du \right)}^{\frac{1}{2}}\\
&+(1-\bar{u})^{3-s}\psi^2(2-2\bar{u}),
\end{aligned}
\end{equation}
which implies as in \eqref{Sobstep3}
\begin{equation}\label{Sob2step3}
\begin{aligned}
&\int_{2-\bar{u}}^{+\infty} u^{\frac{5}{2p}-\frac{3}{4}} {(u-\bar{u})}^{-s} {\phi}^2 du \\
\leq &\int_{\bar{u}}^{+\infty} u^{\frac{5}{2p}-\frac{3}{4}}{(u-\bar{u})}^{2-s}  {\phi_u}^2 du+(1-\bar{u})^{3-s}\psi^2(2-2\bar{u})\\
&\lesssim {\left(M(\phi)(\bar{u}) \right)}^2++(1-\bar{u})^{3-s}\psi^2(2-2\bar{u}),
\end{aligned}
\end{equation}
then Lemma \ref{lem8} follows from \eqref{Sob2step1} and \eqref{Sob2step3}.
\end{proof}

\section{Proof of Theorem \ref{thm1.1}}

In this section we will prove the main result(Theorem \ref{thm1.1}) by global iteration method. For any $\bar{\psi}$ with $\bar{\phi}=r\bar{\psi}$, such that $\sup\limits_{\frac12\leq \bar{u}\leq \bar{U}} M(\bar{\phi})(\bar{u}) \leq M_0 \varepsilon$($M_0$ is a positive constant to be determined), we define a map $\mathcal{F}:\bar{\phi} \to \phi$ such that $\phi$ solves
\begin{equation}\label{fixpoint}
\begin{cases}
\phi_{u\bar{u}}+\frac{\mu(2-\mu)\phi}{4{(u-\bar{u})}^2}=\frac{{|\bar{\phi}|}^p}{{(u-\bar{u})}^{p-1}{(u+\bar{u})}^{\frac{\mu(p-1)}{2}}} \\
t=0:\,  \phi=\eps\psi_0, \phi_t=\eps\psi_1.
\end{cases}
\end{equation}
We want to prove that $\mathcal{F}$ maps the set
$$
\textbf{X}=\{ \phi=r\psi | \psi (0, r)=\eps\psi_0(r), \psi_t(0,r)=\eps\psi_1(r),\sup\limits_{1/2\leq \bar{u} \leq \bar{U}} M(\phi)(\bar{u}) \leq M_0 \varepsilon \}
$$
to itself and is a contraction map, thus, for any $\bar{\phi}_1, \bar{\phi}_2\in \textbf{X}$, it holds that
$$
\sup\limits_{1/2\leq \bar{u} \leq \bar{U}} M(\phi_1-\phi_2)(\bar{u}) \leq \frac{1}{2}\sup\limits_{1/2\leq \bar{u} \leq \bar{U}} M(\bar{\phi}_1-\bar{\phi}_2)(\bar{u}).
$$
Then by contraction mapping theorem, $\mathcal{F}$ has a fixed point, which is our global solution. We only prove that $\mathcal{F}$ maps the set $\textbf{X}$ to itself, the contraction can be proved in a similar way.

Take $s=\frac{1}{4}+\frac{1}{2p}$, for $\bar{\phi}\in \textbf{X}$, by Lemma \ref{lem3} one has
\begin{equation}\label{5d}
\begin{aligned}
\sup\limits_{\frac{1}{2} \leq \bar{u} \leq \bar{U}} M(\phi)(\bar{u}) \lesssim & \varepsilon+\int_1^{\bar{U}} {\left( \int_{\bar{u}}^{+\infty} u^{3s} {(u-\bar{u})}^s G^2 du \right)}^{\frac{1}{2}} d\bar{u}\\
&+\int_{\frac{1}{2}}^1 {\left( \int_{2-\bar{u}}^{+\infty} u^{3s} {(u-\bar{u})}^s G^2 du \right)}^{\frac{1}{2}} d\bar{u}\\
\triangleq & \varepsilon+I+II.\\
\end{aligned}
\end{equation}
We first estimate the integral in $I$ for $p_S(3+\mu)<p\le 2$. By Lemma \ref{lem4} one has
\begin{equation}\label{6d}
\begin{aligned}
&\int_{\bar{u}}^{+\infty} u^{3s} {(u-\bar{u})}^s G^2 du \\
=& \int_{\bar{u}}^{+\infty} u^{3s} {(u-\bar{u})}^s{(u-\bar{u})}^{-2(p-1)}{(u+\bar{u})}^{-\mu(p-1)} {|\bar{\phi}|}^{2p} du \\
\leq &{M(\bar{\phi})}^{2(p-1)} \int_{\bar{u}}^{+\infty} u^{3s-\mu(p-1)} {(u-\bar{u})}^{s-2(p-1)} {|\bar{\phi}|}^2 u^{-\frac{2(p-1)}{p}} du \\
=&   {M(\bar{\phi})}^{2(p-1)} \int_{\bar{u}}^{+\infty} u^{-\alpha} {(u-\bar{u})}^{-\beta-1} {|\bar{\phi}|}^2  du,
\end{aligned}
\end{equation}
where for $p_S(3+\mu)<p\le 2$ and $1.5\le \mu<2$
\begin{equation}\label{cond1}
\begin{aligned}
\alpha&=-3(\frac{1}{4}+\frac{1}{2p})+\mu(p-1)+\frac{2}{p}(p-1)\\
&=\frac{(\mu+2)p^2-(\mu+4)p-2}{p}-2p+\frac{21}{4}-\frac{3}{2p}\\
&>0,\\
\beta&=-(\frac{1}{4}+\frac{1}{2p})+2(p-1)-1\ge 0,\\
2&-\beta>0,\\
1&-\beta\ge s,\\
\end{aligned}
\end{equation}
which leads to
\begin{equation}\label{I1}
\begin{aligned}
&\int_{\bar{u}}^{+\infty} u^{-\alpha} {(u-\bar{u})}^{-1-\beta} \phi^2 du\\
=&-\frac{1}{\beta}\int_{\bar{u}}^{+\infty} u^{-\alpha}  \phi^2 d{(u-\bar{u})}^{-\beta} \\
=&-\frac{1}{\beta}\int_{\bar{u}}^{+\infty} \alpha u^{-\alpha-1}  \phi^2 {(u-\bar{u})}^{-\beta}du+\frac{2}{\beta}\int_{\bar{u}}^{+\infty} u^{-\alpha}  \phi \cdot \phi_u {(u-\bar{u})}^{-\beta} du\\
\leq&\frac{2}{\beta}\int_{\bar{u}}^{+\infty} u^{-\alpha}  \phi \cdot \phi_u {(u-\bar{u})}^{-\beta} du\\
\lesssim& {\left(\int_{\bar{u}}^{+\infty} u^{-\alpha}   {(u-\bar{u})}^{-\beta-1} u^2 du \right)}^{\frac{1}{2}}{\left(\int_{\bar{u}}^{+\infty} u^{-\alpha}   {(u-\bar{u})}^{-\beta+1} {\phi_u}^2 du \right)}^{\frac{1}{2}}\\
\lesssim& {\left(\int_{\bar{u}}^{+\infty} u^{-\alpha}   {(u-\bar{u})}^{-\beta-1} u^2 du \right)}^{\frac{1}{2}}{\left(\int_{\bar{u}}^{+\infty} u^{-\alpha-\beta+1-4s}u^{3s} {(u-\bar{u})}^{s} {\phi_u}^{2} du \right)}^{\frac{1}{2}},\\
\end{aligned}
\end{equation}
where we used the fact that for $1-\beta\ge s$
\[
\begin{aligned}
{(u-\bar{u})}^{-\beta+1}&\leq {(u-\bar{u})}^{-\beta+1-s}{(u-\bar{u})}^{s}\\
& \leq u^{1-\beta-s}{(u-\bar{u})}^{s}.
\end{aligned}
\]
It is easy to see
\[
\begin{aligned}
\gamma&\triangleq \alpha+\beta+4s-1\\
&=\mu(p-1)+\frac{2}{p}(p-1)+2(p-1)-2\\
&=\frac{(\mu+2)p^2-(\mu+4)p-2}{p}+2\\
&>2,
\end{aligned}
\]
the by combining
$$
u \geq \bar{u}, \, u^{-\gamma} \leq {\bar{u}}^{-\gamma},
$$
we obtain for $p_S(3+\mu)<p\le 2$ and $1.5\le \mu<2$
\begin{equation}\label{I2}
\begin{aligned}
&\int_{\bar{u}}^{+\infty} u^{3s} {(u-\bar{u})}^s G^2 du \\
\lesssim& \bar{u}^{-\gamma}M(\bar{\phi})^{2p}(\bar{u}).\\
\end{aligned}
\end{equation}

We are going to estimate the integral in $II$ for $p_S(3+\mu)<p\le 2$. By Lemma \ref{lem8} we get
\begin{equation}\label{II1}
\begin{aligned}
&\int_{2-\bar{u}}^{+\infty} u^{3s} {(u-\bar{u})}^s G^2 du \\
=& \int_{2-\bar{u}}^{+\infty} u^{3s} {(u-\bar{u})}^s{(u-\bar{u})}^{-2(p-1)}{(u+\bar{u})}^{-\mu(p-1)} {|\bar{\phi}|}^{2p} du \\
\leq &\left(M(\bar{\phi})^{2(p-1)}+(1-\bar{u})^{(3-s)(p-1)}\psi_0^{2(p-1)}(2-2\bar{u})\right)\\
&\times \int_{2-\bar{u}}^{+\infty} u^{3s-\mu(p-1)-\frac{2(p-1)}{p}} {(u-\bar{u})}^{s-2(p-1)} {|\bar{\phi}|}^2 du \\
\triangleq &\left(M(\bar{\phi})^{2(p-1)}+(1-\bar{u})^{(3-s)(p-1)}\psi_0^{2(p-1)}(2-2\bar{u})\right)\\
&\times \int_{2-\bar{u}}^{+\infty} u^{-\alpha} {(u-\bar{u})}^{-\beta-1} {|\bar{\phi}|}^2 du,\\
\end{aligned}
\end{equation}
where $\alpha, \beta$ is the same as those in \eqref{cond1} and the same conditions are satisfied. Then we may estimate the last term in the above inequality in a similar way as that of \eqref{I1}, and the only difference is that the initial data will appear in this case, thus we have
\begin{equation}\label{II2}
\begin{aligned}
&\int_{2-\bar{u}}^{+\infty} u^{-\alpha} {(u-\bar{u})}^{-1-\beta} \bar{\phi}^2 du\\
\lesssim& (1-\bar{u})^{2-\beta}\psi_0^2(2-2\bar{u})+\int_{2-\bar{u}}^{+\infty} u^{-\alpha} {(u-\bar{u})}^{-1+\beta} \bar{\phi}_u^2 du\\
\lesssim& (1-\bar{u})^{2-\beta}\psi_0^2(2-2\bar{u})+\bar{u}^{-\gamma}M(\bar{\phi})^2,\\
\end{aligned}
\end{equation}
this implies by combining \eqref{II1}
\begin{equation}\label{II3}
\begin{aligned}
&\int_{2-\bar{u}}^{+\infty} u^{3s} {(u-\bar{u})}^s G^2 du \\
\lesssim& \left[M(\bar{\phi})^{2(p-1)}+(1-\bar{u})^{(3-s)(p-1)}\psi_0^{2(p-1)}(2-2\bar{u})\right]\\
&\times \left[\bar{u}^{-\gamma}M(\bar{\phi})^2+(1-\bar{u})^{2-\beta}\psi_0^2(2-2\bar{u})\right].\\
\end{aligned}
\end{equation}

Plugging \eqref{I2} and \eqref{II3} into \eqref{5d}, finally we get for $p_S(3+\mu)<p\le 2$ and $1.5\le \mu<2$
\begin{equation}\label{I6}
\sup\limits_{\frac{1}{2}\leq\bar{u}\leq\bar{U}} M(\phi)(\bar{u}) \leq C_0 \varepsilon+\tilde{C}_0 \varepsilon^p+C_1\sup\limits_{\frac{1}{2}\leq\bar{u}\leq\bar{U}} {M(\bar{\phi})}^p(\bar{u}),
\end{equation}
where $C_0, \tilde{C}_0, C_1$ are some positive constants independent of $\varepsilon$,
and if we take $\varepsilon \leq 1$, then there exists $C_2>0$ such that
\begin{equation}\label{I7}
\sup\limits_{\frac{1}{2}\leq\bar{u}\leq\bar{U}} M(\phi)(\bar{u}) \leq C_2 \varepsilon+C_1\sup\limits_{\frac{1}{2}\leq\bar{u}\leq\bar{U}} {M(\bar{\phi})}^p(\bar{u}).
\end{equation}
Set $M_0=2C_2$, then $\mathcal{F}$ maps $X$ to $X$ provided that
$$
C_1M_0(M_0\varepsilon)^{p-1} \leq C_2.
$$
In a similar way, we can prove $\mathcal{F}$ is a contraction mapping, this finishes the proof of Theorem \ref{thm1.1}.
\section*{Acknowledgement}

Ning-An Lai is supported by NSF of Zhejiang Province(LY18A010008) and NSFC 11771194, Yi Zhou is supported by Key Laboratory of Mathematics for Nonlinear
Sciences (Fudan University), Ministry of Education of China, Shanghai Key Laboratory for Contemporary
Applied Mathematics, School of Mathematical Sciences, Fudan University, NSFC (11421061), 973 program
(2013CB834100) and 111 project.


\end{document}